\documentclass[UTF8]{amsart}
\usepackage{amsmath,amssymb,graphicx,bbm,amsthm}
\usepackage{caption}
\usepackage{subcaption}
\usepackage[american]{babel}
\usepackage{color}
\usepackage{hyperref}
\hypersetup{colorlinks = true, allcolors = blue}
\usepackage[all]{xy}
\newtheorem{theorem}{Theorem}[section]
\newtheorem{lemma}[theorem]{Lemma}
\newtheorem{proposition}[theorem]{Proposition}

\numberwithin{equation}{section}

\theoremstyle{definition}
\newtheorem{definition}[theorem]{Definition}

\theoremstyle{example}

\theoremstyle{remark}
\newtheorem*{remark}{Remark}
\newtheorem*{claim**}{Claim-$\blacksquare$}
\newtheorem*{claim*}{Claim-$\bigstar$}

\DeclareMathOperator{\sys}{sys}

\DeclareMathOperator{\arccosh}{arccosh}
\DeclareMathOperator{\arcsinh}{arcsinh}

\DeclareMathOperator{\area}{Area}
\DeclareMathOperator{\perim}{perim}
\DeclareMathOperator{\prob}{Prob}
\DeclareMathOperator{\vol}{Vol_{WP}}
\DeclareMathOperator{\slfs}{\mathcal{L}_{sys}^{fill}}
\usepackage{xcolor}
\usepackage{verbatim}

\newcommand{\sM}{\mathcal{M}}

\usepackage{listings}
\usepackage{color} 
\definecolor{mygreen}{RGB}{28,172,0} 
\definecolor{mylilas}{RGB}{170,55,241}

\begin{document}

\lstset{language=Matlab,
    breaklines=true,
    morekeywords={matlab2tikz},
    keywordstyle=\color{blue},
    morekeywords=[2]{1}, keywordstyle=[2]{\color{black}},
    identifierstyle=\color{blue},
    stringstyle=\color{mylilas},
    commentstyle=\color{mygreen},
    showstringspaces=false,
    numbers=left,
    numberstyle={\tiny \color{black}},
    numbersep=9pt, 
	xleftmargin=1cm
}

\title[Filling closed multi-geodesics]{Asymptotics of shortest filling closed multi-geodesics}

\author{Yue Gao}
\address[Y. ~G. ]{School of Mathematics and Statistics, Anhui Normal University, Wuhu, Anhui, China}
\email{yuegao@ahnu.edu.cn}

\author{Zhongzi Wang}
\address[Z. ~W. ]{School of Mathematical Sciences, Peking University, Beijing, China}
\email{wangzz22@stu.pku.edu.cn}

\author{Yunhui Wu}
\address[Y. ~W. ]{Tsinghua University \& BIMSA, Beijing, China}
\email{yunhui\_wu@mail.tsinghua.edu.cn} 
\date{}

\maketitle

\begin{abstract}
In this paper, we investigate the asymptotics of shortest filling closed multi-geodesics of closed hyperbolic surfaces as systole $\to 0$ or as genus $\to \infty$. We first show that for a closed hyperbolic surface $X_g$ of genus $g$, the length of a shortest filling closed multi-geodesic of $X_g$ is uniformly comparable to $$\left(g+\sum\limits_{\textit{closed geodesic }\gamma\subset X_g, \ \ell(\gamma)<1}\log \left(\frac{1}{\ell(\gamma)}\right)\right).$$ As an application, we show that as $g\to \infty$, a Weil-Petersson random hyperbolic surface has a shortest closed multi-geodesic of length uniformly comparable to $g$. We also show that this is true for a random hyperbolic surface in the Brooks-Makover model. 
\end{abstract}

\tableofcontents
\section{Introduction}
On a closed hyperbolic surface, a closed multi-geodesic is called \emph{filling} if it cuts the surface into topological disks. The study on filling closed multi-geodesics was initiated in \cite{thurston1988geometry}, which plays a role in the proof of Nielsen realization theorem \cite{kerckhoff1983nielsen}, the construction of pseudo-Anosov mappings by the Dehn twists on a filling pair \cite{thurston1988geometry}, and the construction of Thurston spine \cite{thurston1986spine}. Later on, filling closed multi-geodesics arise from various topics on hyperbolic surfaces, for example, the systole function \cite{schaller1999systoles}, and the asymptotics of filling closed geodesics on random hyperbolic surfaces \cite{wu2025prime, dozier2023counting}.

For a closed hyperbolic surface $X_g$ of genus $g \geq 2$, denote by $\slfs(X_g)$ the minimal length of filling closed multi-geodesics on $X_g$. Let $\mathcal{M}_g$ denote the moduli space of closed hyperbolic surfaces of genus $g$. The quantity
\[\slfs:\sM_g \to \mathbb{R}^{>0}\]
gives a natural function on $\sM_g$, which is proper by the standard Collar Lemma (see \emph{e.g.} \cite[Theorem 4.4.6]{buser2010geometry}). Very recently, it is obtained by \cite{gao2025shortest} that
\[\min\limits_{X_g\in \sM_g}\slfs(X_g)=(8g-4)\arccosh\left(\sqrt{2}\cos\left(\frac{\pi}{8g-4}\right)\right),\]
that is half of the perimeter of the regular right-angled hyperbolic $(8g-4)$-gon. One may also see \cite{Are15, aougab2015minimally, sanki2021conjecture, gaster2021short} and the references therein for related results. 
\subsection{Uniform asymptotics of $\slfs$}
As introduced above, the function $\slfs(X_g)$ is unbounded as $X_g$ approaches the boundary of $\sM_g$. In this work, we study its uniform asymptotic behavior as systole goes to $0$. Our first main result is as follows.
\begin{theorem}\label{mt-1}
Let $X_g$ be a closed hyperbolic surface of genus $g$. Then
\[\pi (g-1)+\mathrm{R}(X_g) \le \slfs(X_g) \le 300g+12\mathrm{R}(X_g)\]
where
\[
	\mathrm{R}(X_g) = \sum\limits_{\textit{closed geodesic }\gamma\subset X_g, \ \ell(\gamma)<1} \log \left(\frac{1}{\ell(\gamma)}\right). 
\]
\end{theorem}
\noindent The lower bound of Theorem \ref{mt-1} follows directly from the Collar Lemma and a two-dimensional isoperimetric inequality. To prove the upper bound, we introduce the dual $4$-valent graph of a filling graph $G\subset X_g$, and show that it is isotopic to a filling closed multi-geodesic if $G$ has minimal length (see Section \ref{sec:dual_gph}).

More precisely, for an embedded filling graph $G\subset X_g$, its \emph{dual $4$-valent graph $\mathcal{D}(G)$} is defined by connecting the mid-points of neighboring edges of $G$ by geodesic segments. This $4$-valent graph on $X_g$ can be treated as a finite set of closed curves. In general, these curves may be contractible or not in minimal position (see Figure \ref{fig:bigon}). However, we are able to show that if $G$ realizes the minimal length among all filling graphs, its dual $4$-valent graph $\mathcal{D}(G)$ is in minimal position and isotopic to a filling set of closed geodesics (see Proposition \ref{prop:ess_mid_gph}). This can be proved by showing the nonexistence of disks, immersed monogons and bigons bounded by curves in $\mathcal{D}(G)$ by using the Gauss-Bonnet formula. 
This depends on the existence and properties of the graphs realizing the minimal length among all finite filling geodesic graphs (see Proposition \ref{prop:min_prop}, which is a consequence of \cite[Theorem 1.2(3)]{martelli2015spines}; see also \emph{e.g.} \cite{meeks1982embedded, white10injectivity, namazi2009heegaard} for graphs/surfaces with minimal lengths/areas in higher dimensional spaces). 
Next, from an elementary inequality (see Proposition \ref{prop:graph_estimate}), it is not hard to see that
\[\slfs(X_g)\leq \ell(\mathcal{D}(G))\leq 2 \ell(G).\]
At this stage, to obtain the upper bound of Theorem \ref{mt-1}, it suffices to construct an explicit embedded filling graph, still denoted by $G$, in $X_g$ such that
\[\ell(G)\leq 150 g+6\mathrm{R}(X_g).\]
This will be guaranteed by Proposition \ref{prop:length_upper} based on \cite[Theorem 4.5.2]{buser2010geometry}. 
\subsection{Random hyperbolic surfaces}
The study of Weil-Petersson random hyperbolic surfaces was started in \cite{Mirz10, guth2011pants, mirzakhani2013growth}. This subject has become very active in recent years. Some surprising geometric properties and its connection to other subjects have been established. For geometric properties on this random surface model, one may refer to \cite{guth2011pants, mirzakhani2013growth, mirzakhani2019lengths, PWX22, MT22, nie2023large, dozier2023counting, he2023non, wu2025prime} and the references therein for related topics. For volume calculation and random theory on the moduli space of flat surfaces and their relation with the Weil-Petersson geometry on $\mathcal{M}_g$, see, for example, \cite{MZ15, aggarwal2020large, masur2024lengths}. 

We now view the quantity $\slfs(X_g)$ as a positive random variable on $\sM_g$. By \cite{gao2025shortest} and the Collar Lemma, it is known that
\[\lim\limits_{g\to \infty}\frac{\min\limits_{X_g\in \sM_g}\slfs(X_g)}{g}=8\arccosh\left(\sqrt{2}\right) \textit{ and } \sup\limits_{X_g\in \sM_g}\slfs(X_g)=\infty.\]
As an application of Theorem \ref{mt-1}, we show that as $g\to \infty$, a generic $X_g \in \sM_g$ has $\slfs(X_g)$ uniformly comparable to $g$. More precisely, let $\prob_{\mathrm{WP}}^{g}$ and $\mathbb{E}_{\mathrm{WP}}^{g}$ be the probability measure and expected value on $\sM_g$ induced by the Weil-Petersson metric, respectively. We prove
\begin{theorem}
	For any $C > 300$, 
	\[
		\lim_{g\to \infty} \prob_{\mathrm{WP}}^{g}\left(X_g \in \mathcal{M}_g; \ 7g < \mathcal{L}_{\sys}^{\mathrm{fill}}(X_g) < Cg\right) =1.
	\]
	\label{thm:wp_prob}
\end{theorem}
\noindent For the expected value, we prove the following.
\begin{theorem}
	For any $p >0$, there exist two positive constants $C_1 = C_1(p)$ and $C_2 = C_2(p)$ only depending on $p$ such that for $g$ large enough,  
	\[
		C_1 <	\frac{ \mathbb{E}_{\mathrm{WP}}^{g} \left[ (\mathcal{L}_{\sys}^{\mathrm{fill}})^p \right] }{g^p} <C_2. 
	\]
If $p=1$, $C_1 = 7$ and $C_2 > 300$ is arbitrary. 		
	\label{thm:wp_exp}
\end{theorem}
\noindent To prove Theorem \ref{thm:wp_prob} and Theorem \ref{thm:wp_exp}, we will apply Mirzakhani's integration formula \cite[Theorem 7.1]{mirzakhani2007simple} to obtain an effective estimate of the expected value $\mathbb{E}_{\mathrm{WP}}^{g} \left[ (\mathrm{R}(X_g))^p \right]$, where $\mathrm{R}(X_g)$ is as in Theorem \ref{mt-1}. See Lemma \ref{lem:r_exp}.\\

We also study the behavior of $\mathcal{L}_{\sys}^{\mathrm{fill}}$ in the Brooks-Makover model \cite{brooks2004random} of random hyperbolic surfaces. Firstly, we briefly recall the construction of random hyperbolic surfaces in the Brooks-Makover model \cite{brooks2004random}. For $N\geq 2$, take $2N$ hyperbolic ideal triangles and glue them together into a cusped hyperbolic surface, where the gluing scheme is given by a random oriented trivalent graph. The glue is required to have zero shearing coordinates. Denote the gluing pattern as $\omega$ and the set of all the gluing patterns as $\Omega_{N}$. For $\omega\in \Omega_N$, the cusped hyperbolic surface obtained by gluing ideal triangles following the pattern $\omega$ is denoted by $S_O(\omega)$. Its compactification by filling in all the cusps is denoted by $S_C(\omega)$. The set $\Omega_N$ is finite, so define the probability of its subset $B\subset \Omega_N$ as 
\[
	\prob_{\mathrm{BM}}^N(B) = \frac{|B|}{|\Omega_N|}. 
\]
One may see \emph{e.g.} \cite{gamburd2006poisson, petri2015random, budzinski2021diameter, SW23, liu2023random} for more details of this random model. Our result on $\slfs$ for this random hyperbolic surface model is as follows.
\begin{theorem}\label{mt-3}
	For any $\omega \in \Omega_N$, 
	\[
	  \pi N \le  \slfs (S_O(\omega)) \le 6N. 
	\]
For the compact case,
\[\lim\limits_{N\to \infty}\prob_{\mathrm{BM}}^N \left( \omega\in \Omega_N; \ \frac{7}{2}N <  \mathcal{L}^{\mathrm{fill}}_{\mathrm{sys}} (S_C(\omega)) < 6N \right) = 1.\]
\end{theorem}
\noindent Here, a closed multi-geodesic is called \emph{filling} in a cusped hyperbolic surface if each component of its complement is homotopic to either a point or a circle around a cusp. For the proof of the compact case of Theorem \ref{mt-3}, it relies on the following proposition that is of its own interest.   
\setcounter{section}{3}
\setcounter{theorem}{13}
\begin{proposition}
	Let $X \in \mathcal{M}_g$ and $\mathcal{T}$ be a geodesic triangulation of $X$. If the degree of each vertex of $\mathcal{T}$ is at least $6$, then its $4$-valent dual graph $\mathcal{D}(\mathcal{T})$ is isotopic to a filling closed multi-geodesic in $X$.  
	\label{prop:ess_tri_mid_gph_intro}
\end{proposition}
\setcounter{section}{1}
\setcounter{theorem}{5}

\subsection*{Plan of the paper}
This paper is organized as follows. In Section~\ref{sec:pre}, some preliminaries are provided. 
In Section~\ref{sec:dual_gph}, we prove that the dual $4$-valent graph of $G$ is in minimal position and isotopic to a filling closed multi-geodesic of the surface when $G$ is the filling graph with minimal length or $G$ is a geodesic triangulation with degree $ \ge 6$ at every vertex. In Section \ref{sec-t-1}, we prove Theorem \ref{mt-1}. In Section \ref{s-rs}, we prove Theorems \ref{thm:wp_prob}, \ref{thm:wp_exp}, and \ref{mt-3}. 

\subsection*{Acknowledgement}
The authors are all grateful to Yuxin He for helpful discussions. Y.~G. is supported by NSFC grant No. 12301082. Z.~W. is supported by NSFC grant No. 125B2006. Y.~W. is partially supported by supported by National Key R \& D Program of China (2025YFA1017500) and NSFC grants No. 12361141813 and 12425107. 

\section{Preliminaries}
\label{sec:pre}

In this section, we provide some necessary background on two-dimensional hyperbolic geometry and surface topology.

\subsection{Filling graphs with minimal length on hyperbolic surfaces}\label{sec:short_gph}

Take a closed hyperbolic surface $X \in \mathcal{M}_g$. Consider a connected graph $G=(V,E)$ embedded in $X$, whose any pair of vertices is joined by at most one edge and whose degree at any vertex is $ \ge 3$. 
We say that $G$ {\em fills} $X$ if $X \backslash G$ consists of topological disks. 
If any $e\in E$ is a rectifiable curve, then one may define the length of $G$ as
\[
	\ell(G) \overset{\mathrm{def}}{=} \sum_{e\in E}\ell(e). 
\]
We say that a graph $G \subset X$ is \emph{geodesic}, if $e\in E$ is a geodesic segment. 
For $X \in \mathcal{M}_g$, define
\[
	\mathcal{G} = \mathcal{G}_X \overset{\mathrm{def}}{=} \{ G \subset X \text{ a finite filling geodesic graph} \}. 
\] 
The following proposition concerns the graph in $\mathcal{G}$ with minimal length, which is a direct consequence of \cite[Theorem 1.2(3)]{martelli2015spines} by Martelli-Novaga-Pluda-Riolo.
\begin{proposition}
There exists an embedded graph $G_0\in \mathcal{G}$ such that $$\ell(G_0)=\min\limits_{G \in \mathcal{G}} \ell(G).$$ 

	Moreover, for such $G_0$, 
	\begin{enumerate}
		\item $G_0$ is trivalent. \label{item:trivalent}
		\item At each vertex, the angle between any two edges equals $ \frac{2}{3} \pi$. \label{item:hyperplane}

		\item $X \backslash G_0$ consists of exactly one topological disk.
	\end{enumerate}
	\label{prop:min_prop}
\end{proposition}

For $X\in \sM_g$, as in \cite{martelli2015spines}, an embedded finite graph $G \subset X$ is called a \emph{spine} if the complement $X \backslash G$ of $G$ is a topological disk. The following result is a special case of \cite[Theorem 1.2(3)]{martelli2015spines} in which $X$ is not necessarily to be hyperbolic.
\begin{theorem}[{\cite[Theorem 1.2(3)]{martelli2015spines}}]
	Any $X\in \mathcal{M}_g$ has a spine $\Gamma$ of minimal length. The spine $\Gamma$ is a trivalent geodesic graph. At each vertex, the angle between any two edges equals $ \frac{2}{3} \pi$.
	\label{thm:spine}
\end{theorem}

Before proving Proposition \ref{prop:min_prop}, we first show the following lemma.
\begin{lemma}
	For any $G \in \mathcal{G}_X$, there is a spine $\Gamma$ such that 
	\[
		\ell(\Gamma) \le \ell(G). 
	\]
	\label{lem:spine_compare}
\end{lemma}
\begin{proof}
	If $G$ cuts $X$ into only one disk, then the lemma follows immediately since $G$ itself is a spine. 

	If $G$ cuts $X$ into more than one disks, then there is at least one edge of $G$, on whose two sides are different disks. As illustrated in Figure \ref{fig:two_disk}, the edge $e$ seperates the disks $D_1$ and $D_2$. If we remove $e$, then we get a new graph $G_1$, a filling piecewise geodesic graph with smaller length and less faces than $G$. Continuously removing edges that are adjacent to two distinct faces, one may get a piecewise geodesic filling graph with exactly one face. Perturbing this graph to a smooth filling graph, one may get the spine $\Gamma$, whose length is smaller than $G$. The proof is complete. 
\end{proof}
\begin{figure}[htbp]
	\centering
	\begin{subfigure}[htbp]{.45\textwidth}
	\begin{center}
		\includegraphics{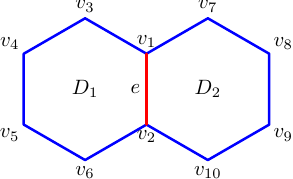}
	\end{center}
	\caption{$D_1$, $D_2$ in $G$}
	\label{fig:two_disk_1}
	\end{subfigure}
	\begin{subfigure}[htbp]{.45\textwidth}
	\begin{center}
		\includegraphics{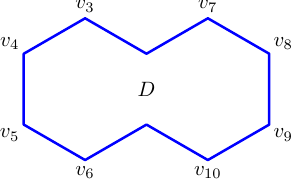}
	\end{center}
	\caption{$D$ in $G_1$}
	\label{fig:two_disk_2}
	\end{subfigure}
	
	\caption{Modification near a common edge of two disks}
	\label{fig:two_disk}
\end{figure}	

Now we are ready to prove Proposition \ref{prop:min_prop}.
\begin{proof}[Proof of Proposition \ref{prop:min_prop}]
	By Lemma \ref{lem:spine_compare}, for any graph $G\in\mathcal{G}$, there is a spine $\Gamma$ such that $\ell(\Gamma) \le \ell(G)$. By Theorem \ref{thm:spine}, the spine with minimal length is a graph $G_0$ in $\mathcal{G}$, hence $G_0$ is the graph in $\mathcal{G}$ with minimal length. The graph $G_0$ is trivalent, cuts the surface into one face. At each vertex of $G_0$, the angle between any two edges equals $ \frac{2}{3}\pi$. The proof is complete. 
\end{proof}

\subsection{Trigon decomposition}
Let us begin with the following definition.
\begin{definition}[{\cite[Definition 4.5.1]{buser2010geometry}}]
	In a closed hyperbolic surface $X$, a {\em trigon} is defined as either an embedded geodesic triangle in $X$ or an annulus in $X$ whose one boundary component is a simple closed geodesic, and the other boundary component consists of two geodesic arcs (see Figure \ref{fig:trigon}). In the latter case, the closed geodesic and the two arcs are the {\em sides} of the trigon. 
	
	For any trigon, its {\em perimeter} is the sum of the lengths of its three sides. 
	\label{def:trigon}
\end{definition}

\begin{figure}[htbp]
	\centering
	\includegraphics{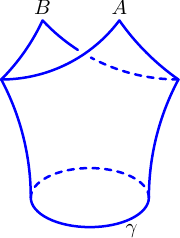}
	\caption{Annulus shaped trigon}
	\label{fig:trigon}
\end{figure}

\begin{theorem}[{\cite[Theorem 4.5.2]{buser2010geometry}}]
	Any closed hyperbolic surface admits a trigon decomposition such that all trigons have sides length $ \le \log 4$ and area between $0.19$ and $1.36$. In the construction of trigon decomposition, any simple closed geodesic with length $ \le \log 4$ is a side of an annulus-shaped trigon and has a neighborhood consisting of a pair of annulus-shaped trigons. 
	\label{thm:buser_trigon}
\end{theorem}

\begin{remark}
	By the Collar Lemma, the set of all closed geodesics with length $ \le \log 4$ in a hyperbolic surface consists of pairwise disjoint simple closed geodesics. 
\end{remark}
For any annulus-shaped trigon shown in Figure \ref{fig:trigon}, it is known from \emph{e.g.} the proof of \cite[Theorem 4.5.2]{buser2010geometry} that the distance between the vertex $A$ or $B$ and the simple closed geodesic side $\gamma$ is 
\begin{equation}\label{rmk:trigonlength}
 d(A, \gamma) = d(B, \gamma) = \arcsinh \left( \frac{1}{\sinh \left( \frac{1}{2}\ell(\gamma) \right) } \right) - \frac{1}{2}\log 2.    
\end{equation}

\subsection{Immersion}
For clarity of this paper, we recall some definitions on surface topology. 

\begin{definition}
	For a curve in $ \mathbb{R}^2$ with end points having exactly one self-intersection point, according to the Jordan curve theorem, it bounds a disk with one corner (see Figure \ref{fig:monogon_def}). We call this disk \emph{monogon}. For two curves in $ \mathbb{R}^2$ with end points having exactly two self-intersection points, according to the Jordan curve theorem, they also bound a disk with two corners (see Figure \ref{fig:bigon_def}). We call this disk \emph{bigon}.
	\label{def:mono_bi}
\end{definition}

\begin{figure}[htbp]
	\centering
	\begin{subfigure}[htbp]{.45\textwidth}
	\begin{center}
	\includegraphics{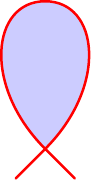}
	\end{center}
	\caption{A monogon}
	\label{fig:monogon_def}
	\end{subfigure}
	\begin{subfigure}[htbp]{.45\textwidth}
	\begin{center}
		\includegraphics{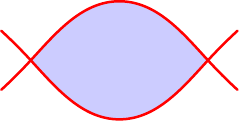}
	\end{center}
	\caption{A bigon}
	\label{fig:bigon_def}
	\end{subfigure}
	\caption{A monogon and a bigon}
	\label{fig:mono_bi}
\end{figure}

\begin{definition}
	Let $G$ be a monogon or a bigon, $X$ be a closed surface, and $f:G\to X$ be a map. We call $G$ is \emph{immersed} in $X$ if $f$ is locally injective; $G$ is \emph{embedded} in $X$ if $f$ is injective. 
	\label{def:immerse}
\end{definition}

In the following, an immersed (or {\em resp.} embedded) monogon ({\em resp.} bigon) on a surface is referred to as the image of an immersed ({\em resp.} embedded) map from a monogon ({\em resp.} bigon) to the surface (see Figure \ref{fig:bigon_immersed} for an immersed bigon). 

\begin{figure}[htbp]
	\centering
	\includegraphics{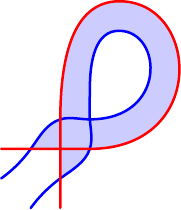}
	\caption{An immersed bigon}
	\label{fig:bigon_immersed}
\end{figure}

\subsection{Curves in minimal position}
Now we recall the minimal position for curves.
\begin{definition}[{\cite[Section 1.2.3, page 30]{farb2011primer}}]
	Let $\Gamma = \{\gamma_1, \gamma_2, ..., \gamma_n\}$ be a set of incontractible closed curves in a surface. Say $\Gamma$ is {\em in minimal position} if for any $i, j = 1,2,..., n$, the number of intersection points of $\gamma_i$ and $\gamma_j$ is minimal in the isotopy classes of $\gamma_i$ and $\gamma_j$. In other words, if $\Gamma$ is in minimal position, then no self-intersection point of $\Gamma$ can be eliminated by homotopy. 
	\label{def:min_position}
\end{definition}

A criterion for curves in minimal position is 
\begin{lemma}[{\cite[Theorem 3.5]{hass1985intersections}} or {\cite[Lemma 5.1]{gaster2017infima}}]
	In a surface, a curve set $\Gamma = \{\gamma_1, \gamma_2, ..., \gamma_n\}$ is in minimal position if and only if no $\gamma_i$ bounds an immersed monogon and no pair $\gamma_i, \gamma_j$ bounds an immersed bigon. 
	\label{lem:min_pos_crit}
\end{lemma}

We recall the following two properties for contractible closed curves and curves in minimal position.
\begin{theorem}[{\cite[Theorem 2.7]{hass1985intersections}}]
A contractible closed curve either bounds a disk or bounds an embedded monogon or an embedded bigon. 
	\label{thm:contractible}
\end{theorem}

\begin{theorem}[{\cite[Theorem 2.1 and 2.2]{hass1994shortening}} or {\cite[Section 2.1]{levitt2000whitehead}}]
	For $X \in \mathcal{M}_g$ and a finite set $\Gamma\subset X$ of closed curves in minimal position, there is an isotopy of $X$, mapping $\Gamma$ to a set of closed geodesics on $X$ up to finitely many triangle moves. A triangle move is illustrated in Figure \ref{fig:triangle_move}. 
	\label{thm:isotopy_geod}
\end{theorem}
\begin{figure}[htbp]
	\centering
	\includegraphics{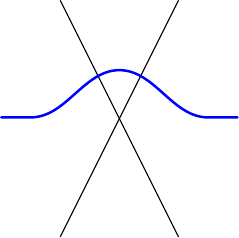}
	\includegraphics{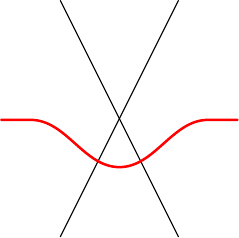}
	\caption{A triangle move}
	\label{fig:triangle_move}
\end{figure}

\section{From shortest filling graphs to filling closed multi-geodesics}\label{sec:dual_gph}
In this section, a filling closed multi-curve is constructed using the shortest filling graph in Section \ref{sec:short_gph}. The length of the closed multi-curve is bounded from above by the length of the graph. 

\begin{definition}
Let $G \in \mathcal{G}$ be a graph embedded in $X\in \mathcal{M}_g$, with geodesic edges, satisfying that $X \backslash G$ are convex polygons. 
Join the mid-points of neighboring edges of $G$ by geodesics, then we get a $4$-valent graph $\mathcal{D}(G)$.	
We define $\mathcal{D}(G)$ as the {\em dual $4$-valent graph} of $G$. 
	\label{def:dual_4_valent_graph}
\end{definition}

\begin{figure}[htbp]
	\centering
	\includegraphics{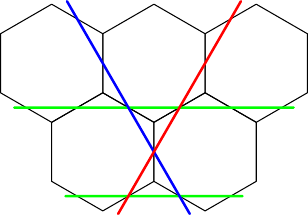}
	\caption{$G$ and its dual $4$-valent graph $\mathcal{D}(G)$}
	\label{fig:midptgph}
\end{figure}

\noindent A $4$-valent graph embedded in a surface can be treated as a union of closed curves on $X$ (see Figure \ref{fig:midptgph} for an illustration). In general, the $4$-valent graph $\mathcal{D}(G)$ may contain contractible curves and may not be in its minimal position (Figure \ref{fig:bigon}). However, 
\begin{proposition}
If $G_0$ is the graph that realizes the minimal length in $\mathcal{G}$, then
   \begin{enumerate}
        \item each closed curve in its dual $4$-valent graph  $\mathcal{D}(G_0)$ is homotopically non-trivial.
        \item The curves in $\mathcal{D}(G_0)$ are in minimal position. 
        \item The dual $4$-valent graph $\mathcal{D}(G_0)$ is isotopic to a filling set of closed geodesics in $X$. 
    \end{enumerate}
	\label{prop:ess_mid_gph}
\end{proposition}

\begin{figure}[htbp]
	\centering
	\includegraphics{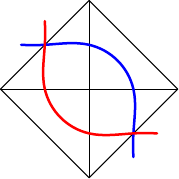}
	\caption{A bigon in $\mathcal{D}(G)$}
	\label{fig:bigon}
\end{figure}

\noindent We prove it in steps. Denote the universal cover of the surface $X$ as $ \widetilde{X}$. For $G\in \mathcal{G}$, denote the pre-image of $G$ and $ \mathcal{D}(G)$ on $ \widetilde{X}$ as $ \widetilde{G}$ and $\mathcal{D}(\widetilde{G})$ respectively. 

\begin{lemma}
	For $G\in \mathcal{G}$, 

	(1) If a curve $\alpha$ in $\mathcal{D}(G)$ bounds an immersed monogon, then the lift $ \widetilde{\alpha}$ of $\alpha$ in $\mathcal{D}(\widetilde{G})$ bounds an embedded monogon. 

	(2) If two curves $\alpha$, $\beta$ in $\mathcal{D}(G)$ bound an immersed bigon, then either one of the lifts $ \widetilde{\alpha}$, $ \widetilde{\beta}$ bounds an embedded monogon, or $ \widetilde{\alpha}$ and $ \widetilde{\beta}$ bound an embedded bigon. 
	
	\label{lem:lift_embed}
\end{lemma}
\begin{proof}
	(1) If a curve $\alpha \subset \mathcal{D}(G) $ bounds an immersed monogon, then any lift of it has at least one self-intersection point. Take one of its lift $ \widetilde{\alpha}: (0,1) \to \widetilde{X}$. Let $t_1, t_2 \in (0,1)$ ($t_1<t_2$) represent its first self-intersection point. More precisely, let $ \widetilde{\alpha}$ be injective in $(0,t_2)$ and $ \widetilde{\alpha}(t_1) = \widetilde{\alpha}(t_2)$. Then  the image $ \widetilde{\alpha}([t_1,t_2])$ bounds an embedded monogon in $ \widetilde{X}$. 

	(2) If either $ \widetilde{\alpha}$ or $ \widetilde{\beta}$ has self-intersections, then it bounds an embedded monogon, as is proved in (1). Assume $ \widetilde{\alpha}$ and $ \widetilde{\beta}$ are simple arcs. Since $\alpha$ and $\beta$ bounds an immersed bigon, the two corners of this immersed bigon are lifted to intersection points of $ \widetilde{\alpha}$ and $ \widetilde{\beta}$ by the homotopy lifting theorem, hence $ \widetilde{\alpha}$ and $ \widetilde{\beta}$ intersect at least twice. The subarcs of $ \widetilde{\alpha}$ and $ \widetilde{\beta}$ between two neighboring intersection points of $ \widetilde{\alpha}$ and $ \widetilde{\beta}$ bound an embedded bigon in $ \widetilde{X}$. 

	The proof is complete. 
\end{proof}

\begin{proposition}
If $G_0$ is the graph that realizes the minimal length in $\mathcal{G}$, then in $\mathcal{D}( \widetilde{G_0})$, each curve does not bound any disk or embedded monogon. Each pair of curves does not bound any embedded bigon. 
	\label{prop:ess_mid_gph_univ}
\end{proposition}
\begin{proof}
	Suppose that a curve $\gamma \subset \mathcal{D}( \widetilde{G_0})$ bounds a disk in $ \widetilde{X}$. Pick the edges in $ \widetilde{G_0}$ intersecting $\gamma$. Then these edges form a cycle and this cycle bounds a disk $D \subset \widetilde{X}$. By Proposition \ref{prop:min_prop}, edges of $G_0$ are geodesics and neighboring edges at a vertex have angle $ \frac{2}{3}\pi$. These are also true for $ \widetilde{G_0}$. Therefore, $D$ is a geodesic polygon (see Figure \ref{fig:disk}). By the construction of $\mathcal{D}( \widetilde{G_0})$, $D$ is a polygon with even number of vertices. The inner angle at each vertex is either $ \frac{2}{3}\pi$ or  $ \frac{4}{3}\pi$. The $ \frac{2}{3}\pi$-angle and the $ \frac{4}{3}\pi$-angle appear alternatively when we travel along $ \partial D$. 
	\begin{figure}[htbp]
		\centering
		\includegraphics{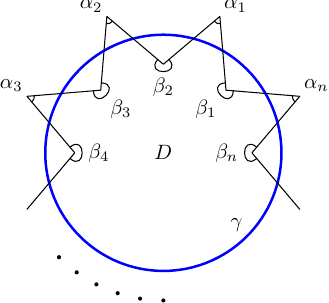}
		\caption{$D$ and $\gamma$}
		\label{fig:disk}
	\end{figure}

	\noindent Let $\alpha_1, \alpha_2, ..., \alpha_n$ be the inner angles of $D$ with value $ \frac{2}{3}\pi$ and $\beta_1, \beta_2, ..., \beta_n$  be the inner angles of $D$ with value $ \frac{4}{3}\pi$. Then by the Gauss-Bonnet formula,
	\begin{eqnarray*}
		\area(D) &=& \sum_{i=1}^n (\pi-\alpha_i) +\sum_{i=1}^n (\pi - \beta_i)  -2\pi\\
			 &=& n \left(\pi - \frac{2}{3}\pi \right) + n \left(\pi - \frac{4}{3}\pi\right) - 2\pi \\
			 &=& -2\pi, 
	\end{eqnarray*}
	which is impossible. Hence no $\gamma \subset \mathcal{D}( \widetilde{G_0})$ bounds a disk. 

	Suppose that $\gamma\subset \mathcal{D}( \widetilde{G_0})$ bounds a monogon. Pick the edges of $ \widetilde{G_0}$ intersecting the monogon. Among these edges, there is one edge passing through the self intersection point of the monogon. There is a disk $D \subset \widetilde{X}$, either bounded by all of these edges (see Figure \ref{fig:monogon_1}) or bounded by all the edges except the edge passing through the self intersection of the monogon (see Figure \ref{fig:monogon_2}). 

\begin{figure}[htbp]
	\centering
	\begin{subfigure}[htbp]{.45\textwidth}
	\begin{center}
		\includegraphics{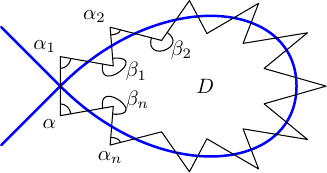}
	\end{center}
	\caption{Case 1}
	\label{fig:monogon_1}
	\end{subfigure}
	\begin{subfigure}[htbp]{.45\textwidth}
	\begin{center}
		\includegraphics{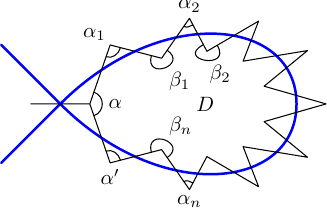}
	\end{center}
	\caption{Case 2}
	\label{fig:monogon_2}
	\end{subfigure}
	
	\caption{The monogon and $D$}
	\label{fig:monogon}
\end{figure}	

For the former case (see Figure \ref{fig:monogon_1}), 
	\begin{eqnarray*}
		\area(D) &=& (\pi - \alpha) + \sum_{i=1}^n (\pi - \alpha_i) + \sum_{i=1}^n (\pi - \beta_i) - 2\pi \\
			 &=& \left( \pi - \frac{2}{3} \pi \right) + \sum_{i=1}^n \left( \pi - \frac{2}{3} \pi \right) + \sum_{i=1}^n \left( \pi - \frac{4}{3} \pi \right) - 2\pi \\
			 &=& - \frac{5}{3} \pi, 
	\end{eqnarray*}
	which is impossible. 

	For the latter case (see Figure \ref{fig:monogon_2}), 
	\begin{eqnarray*}
		\area(D) &=& (\pi - \alpha) + (\pi - \alpha') + \sum_{i=1}^n (\pi - \alpha_i) + \sum_{i=1}^n (\pi - \beta_i) - 2\pi \\
			 &=& \left( \pi - \frac{2}{3} \pi \right) + \left( \pi - \frac{2}{3} \pi \right) + \sum_{i=1}^n \left( \pi - \frac{2}{3} \pi \right) + \sum_{i=1}^n \left( \pi - \frac{4}{3} \pi \right) - 2\pi \\
			 &=& - \frac{4}{3} \pi, 
	\end{eqnarray*}
	which is also impossible. Therefore no curve in $\mathcal{D}(\widetilde{G_0})$ bounds a monogon. 

	Suppose that a pair of curves in $\mathcal{D}(G_0)$ bounds a bigon. Edges of $ \widetilde{G_0}$ intersecting this bigon bound a disk $D$. Then we use Gauss-Bonnet formula to calculate $\area(D)$ to get a contradiction. According to the directions of edges passing through intersection points of the bigon, we classify $D$ into three types shown in Figure \ref{fig:bigon_p} and calculate $\area(D)$ case by case. 

\begin{figure}[htbp]
	\centering
	\begin{subfigure}[htbp]{.45\textwidth}
	\begin{center}
		\includegraphics{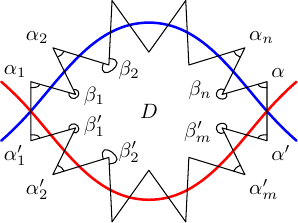}
	\end{center}
	\caption{Case 1}
	\label{fig:bigon_p_1}
	\end{subfigure}
	\begin{subfigure}[htbp]{.45\textwidth}
	\begin{center}
		\includegraphics{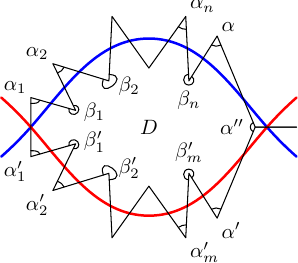}
	\end{center}
	\caption{Case 2}
	\label{fig:bigon_p_2}
	\end{subfigure}

	\begin{subfigure}[htbp]{.45\textwidth}
	\begin{center}
		\includegraphics{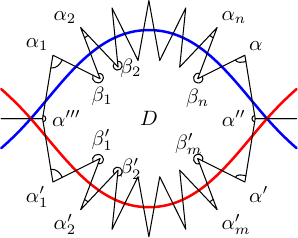}
	\end{center}
	\caption{Case 3}
	\label{fig:bigon_p_3}
	\end{subfigure}
	
	\caption{The bigon and $D$}
	\label{fig:bigon_p}
\end{figure}	
Case 1 (see Figure \ref{fig:bigon_p_1}):
	\begin{eqnarray*}
		\area(D) &= & (\pi-\alpha) + (\pi - \alpha') + \sum_{i=1}^n ( \pi-\alpha_i ) + \sum_{i=1}^n ( \pi-\beta_i ) \\
			 & & + \sum_{j=1}^m ( \pi-\alpha'_j ) + \sum_{j=1}^m ( \pi-\beta'_j ) -2\pi \\
			 &=& \left( \pi - \frac{2}{3}\pi \right) +  \left( \pi - \frac{2}{3}\pi \right) + \sum_{i=1}^n \left( \pi - \frac{2}{3}\pi \right) + \sum_{i=1}^n \left( \pi - \frac{4}{3}\pi \right) \\
			 & &+ \sum_{j=1}^m \left( \pi - \frac{2}{3}\pi \right) + \sum_{j=1}^m \left( \pi - \frac{4}{3}\pi \right) - 2\pi \\
			 &=& - \frac{4}{3}\pi,
	\end{eqnarray*}
	which is impossible. 

	Case 2 (see Figure \ref{fig:bigon_p_2}):
	\begin{eqnarray*}
		\area(D) &= & (\pi-\alpha) + (\pi - \alpha') + (\pi - \alpha'') + \sum_{i=1}^n ( \pi-\alpha_i ) + \sum_{i=1}^n ( \pi-\beta_i ) \\
			 & &+ \sum_{j=1}^m ( \pi-\alpha'_j ) + \sum_{j=1}^m ( \pi-\beta'_j ) -2\pi \\
			 &=& \left( \pi - \frac{2}{3}\pi \right) +  \left( \pi - \frac{2}{3}\pi \right) + \left( \pi - \frac{2}{3}\pi \right) \\
			 & & + \sum_{i=1}^n \left( \pi - \frac{2}{3}\pi \right) + \sum_{i=1}^n \left( \pi - \frac{4}{3}\pi \right) \\
			 & &+ \sum_{j=1}^m \left( \pi - \frac{2}{3}\pi \right) + \sum_{j=1}^m \left( \pi - \frac{4}{3}\pi \right) - 2\pi \\
			 &=& - \pi,
	\end{eqnarray*}
	which is impossible. 

	Case 3 (see Figure \ref{fig:bigon_p_3}):
	\begin{eqnarray*}
		\area(D) &=& (\pi-\alpha) + (\pi - \alpha') + (\pi - \alpha'') + (\pi - \alpha''') \\ 
			 & & + \sum_{i=1}^n ( \pi-\alpha_i ) + \sum_{i=1}^n ( \pi-\beta_i ) + \sum_{j=1}^m ( \pi-\alpha'_j ) + \sum_{j=1}^m ( \pi-\beta'_j ) -2\pi \\
			 &=& \left( \pi - \frac{2}{3}\pi \right) +  \left( \pi - \frac{2}{3}\pi \right) + \left( \pi - \frac{2}{3}\pi \right) + \left( \pi - \frac{2}{3}\pi \right) \\
			 & & + \sum_{i=1}^n \left( \pi - \frac{2}{3}\pi \right) + \sum_{i=1}^n \left( \pi - \frac{4}{3}\pi \right) \\
			 & &+ \sum_{j=1}^m \left( \pi - \frac{2}{3}\pi \right) + \sum_{j=1}^m \left( \pi - \frac{4}{3}\pi \right) - 2\pi \\
			 &=& - \frac{2}{3} \pi,
	\end{eqnarray*}
	which is impossible. Therefore no pair of curves in $\mathcal{D}(\widetilde{G_0})$ bounds a bigon. 

	The proof is complete. 
\end{proof}

\begin{lemma}
	Each closed curve in $\mathcal{D}(G_0)$ is homotopically non-trival. 
	\label{lem:mid_gph_incontractible}
\end{lemma}
\begin{proof}
	Suppose for contradiction that a closed curve $\gamma\subset \mathcal{D}(G_0)$ is contractible, then its lift $ \widetilde{\gamma} \subset \mathcal{D}( \widetilde{G_0})$ is a closed curve in $ \widetilde{X}$. So $ \widetilde{\gamma} $ is either a simple closed curve or a closed curve with self-intersections.  
	 If $ \widetilde{\gamma} $ has self-intersections, then it bounds an embedded monogon or bigon with itself by Theorem \ref{thm:contractible}. 
If $ \widetilde{\gamma} $ is a simple closed curve in $ \widetilde{X}$, then by the Jordan curve theorem, it bounds a disk in $ \widetilde{X}$.
	By Proposition \ref{prop:ess_mid_gph_univ}, each curve in $\mathcal{D}( \widetilde{G_0})$ does not bound any embedded disk or monogon or bigon with itself. This leads to a contradiction. 
	The proof is complete. 
\end{proof}

\begin{lemma}
	The dual $4$-valent graph $\mathcal{D}(G_0)$ of $G_0$ is in minimal position. 
	\label{lem:mid_gph_min_pos}
\end{lemma}
\begin{proof}
	By Lemma \ref{lem:mid_gph_incontractible}, every closed curve in $\mathcal{D}(G_0)$ is  homotopically non-trival. 
	Suppose for contradiction that $\mathcal{D}(G_0)$ is not in minimal position. By Lemma \ref{lem:min_pos_crit}, this is equivalent to saying that there is either a $\alpha \in \mathcal{D}(G_0)$ bounding an immersed monogon, or a pair $\alpha,\beta \in \mathcal{D}(G_0)$ bounding an immersed bigon. On the other hand, by Lemma \ref{lem:lift_embed}, there is either a lift of $\alpha$ bounding an embedded monogon, or a lift of the pair $\alpha$ and $\beta$ bounding an embedded monogon or bigon. 
	But this leads to a contradiction to Proposition \ref{prop:ess_mid_gph_univ}. So $\mathcal{D}(G_0)$ is in minimal position and this lemma is proved. 
\end{proof}

\begin{proposition}
	The dual $4$-valent graph $\mathcal{D}(G_0)$ of $G_0$ is isotopic to a filling set of closed geodesics on $X$. 
	\label{prop:fill_geod}
\end{proposition}
\begin{proof}
	A triangle move to curves in $\mathcal{D}(G_0)$ (illustrated in Figure \ref{fig:triangle_move}) is an isotopy of the curves on $X$. On the other hand, observe that $\mathcal{D}(G_0)$ cuts $X$ into disks. 
	A triangle move to $\mathcal{D}(G_0)$ does not change the topology of the complementary disks because it cuts or pastes a disk along an arc in the disk's boundary to some disks in $X \backslash\mathcal{D}(G_0)$ (see Figure \ref{fig:triangle_move_3}). Then by Theorem \ref{thm:isotopy_geod} we have that $\mathcal{D}(G_0)$ is isotopic to a filling set of closed geodesics on $X$. The proof is complete. 
	\begin{figure}[htbp]
		\centering
		\includegraphics{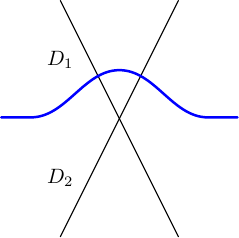}
		\includegraphics{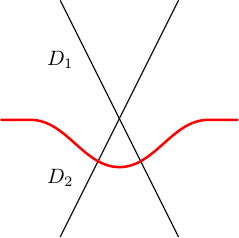}
		\includegraphics{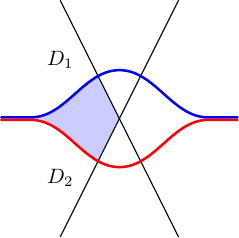}
		\caption{The triangle move cuts the shadowed disk from $D_2$ and pastes it to $D_1$}
		\label{fig:triangle_move_3}
	\end{figure}
\end{proof}

Now we are ready to finish the proof of Proposition \ref{prop:ess_mid_gph}.

\begin{proof}[Proof of Proposition \ref{prop:ess_mid_gph}]
	The result clearly follows from Lemma \ref{lem:mid_gph_incontractible}, Lemma \ref{lem:mid_gph_min_pos} and Proposition \ref{prop:fill_geod}.
\end{proof}

Now we give an estimate to $\ell(\mathcal{D}(G_0))$ in terms of $\ell(G_0)$. We first provide the following elementary inequality. 
\begin{lemma}
	For a geodesic triangle $\triangle ABC$ and a point $P$ in the interior of $\triangle ABC$ (see Figure \ref{fig:fermat_2}), then
	\[
		|AB|+|BC|+|CA| < 2(|PA|+|PB|+|PC|),
	\]
    where $|\cdot|$ means distance.
	\label{lem:triangle_estimate}
\end{lemma}
\begin{figure}[htbp]
	\centering
	\includegraphics{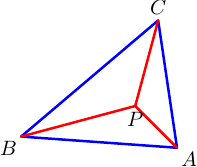}
	\caption{$\triangle ABC$ and $P$}
	\label{fig:fermat_2}
\end{figure}
\begin{proof}
	In $\triangle PAB$, $\triangle PBC$ and $\triangle PCA$, we have
	\begin{eqnarray*}
		|AB| &<& |PA|+|PB| \\
		|BC| &<& |PB|+|PC| \\
		|CA| &<& |PC|+|PA|. 
	\end{eqnarray*}
	Summing the three inequalities up, we prove the lemma. 
\end{proof}
\begin{proposition}
    Let $G_0$ be a graph that realizes the minimal length of graphs in $\mathcal{G}$, and let $\mathcal{D}(G_0)$ be the dual $4$-valent graph of $G_0$. Then 
	\[
		\ell(G_0)<\ell(\mathcal{D}(G_0)) < 2\ell(G_0). 
	\]
	\label{prop:graph_estimate}
\end{proposition}
\begin{proof}
	It is clear that $\ell(G_0)<\ell(\mathcal{D}(G_0))$. It remains to show the RHS inequality.
	Recall that by Proposition \ref{prop:min_prop}, we know that $G_0$ is a trivalent graph that cuts $X$ into one polygon, and the angle at each vertex is $ \frac{2}{3}\pi$. Then $X \backslash \mathcal{D}(G_0)$ consists of a polygon $D$ and certain geodesic triangles $T_1$, $T_2$,..., $T_n$ (see \emph{e.g.} Figure \ref{fig:midptgph}). The number $n$ of geodesic triangles is equal to $|V(G_0)|$, and each vertex $v_i \in V(G_0)$ is the Fermat point of $T_i$. Treat $\mathcal{D}(G_0)$ as a $4$-valent graph in $X$. Observe that for each $e \in E(\mathcal{D}(G_0))$, the face on one side of $e$ is the polygon $D$, and the other side is one of the geodesic triangles  $T_1$, $T_2$, ..., $T_n$. Therefore, $$\ell(\mathcal{D}(G_0)) = \sum_{i=1}^n \perim(T_i),$$ where $\perim(T_i)$ is the perimeter of $T_i$. On the other hand, $G_0$ is disjoint from the interior of $D$. Restricted to each $T_i$, $G_0$ consists of a vertex $v_i$, the Fermat point of $T_i$, and three edges joining $v_i$ and the three vertices of $T_i$, respectively. Denote the sum of these three edges as $f(T_i)$. Then $\ell(G_0) = \sum_{i=1}^n f(T_i)$. Therefore, by Lemma \ref{lem:triangle_estimate} we have
	\begin{eqnarray*}
		\ell(\mathcal{D}(G_0)) = \sum_{i=1}^n \perim(T_i) < 2 \sum_{i=1}^n f(T_i) = 2\ell(G_0).
	\end{eqnarray*}
	The proof is complete. 
\end{proof}

Recall that a geodesic triangulation of a closed hyperbolic surface consists of geodesic triangles such that each of them bounds a disk. For the remainder of this section, we prove the following result on the triangulation of hyperbolic surfaces similar to Proposition \ref{prop:ess_mid_gph}. This will be applied to study the asymptotic behavior of $\slfs$ on Brooks-Makover random surfaces, \emph{i.e.,}  prove Theorem \ref{mt-3}.
\begin{proposition}
	Let $X \in \mathcal{M}_g$ and $\mathcal{T}$ be a geodesic triangulation of $X$. If the degree of every vertex of $\mathcal{T}$ is at least $6$, then 
    \begin{enumerate}
        \item each closed curve in its dual $4$-valent graph  $\mathcal{D}(\mathcal{T})$ is homotopically non-trivial.
        \item The curves in $\mathcal{D}(\mathcal{T})$ are in minimal position. 
        \item The dual $4$-valent graph $\mathcal{D}(\mathcal{T})$ is isotopic to a filling set of closed geodesics in $X$. 
    \end{enumerate}
	\label{prop:ess_tri_mid_gph}
\end{proposition}

The core of the proof of Proposition \ref{prop:ess_tri_mid_gph} is the following proposition, similar to Proposition \ref{prop:ess_mid_gph_univ}. For $X \in \mathcal{M}_g$ and a geodesic triangulation $\mathcal{T}$ of $X$, denote the universal cover of $X$ as $ \widetilde{X}$ and denote the pre-image of $\mathcal{T}$ and $\mathcal{D}(\mathcal{T})$ as $ \widetilde{\mathcal{T}}$ and $\mathcal{D}( \widetilde{\mathcal{T}})$, respectively. Then we prove
\begin{proposition}
	Assume $\mathcal{T}$ is a geodesic triangulation of $X$ that each vertex of $\mathcal{T}$ has degree at least $6$, then in $\mathcal{D}( \widetilde{\mathcal{T}})$, each closed curve does not bound any disk or embedded monogon, and each pair of curves does not bound any embedded bigon. 

	\label{prop:ess_tri_mid_gph_univ}
\end{proposition}
\begin{proof}
	Suppose that a closed curve $\gamma \subset \mathcal{D}(\mathcal{ \widetilde{T}})$ bounds a disk. All triangles in $\mathcal{ \widetilde{T}}$ fully contained in this disk form a triangulated disk $D$ (the disk bounded by the path $v_1v_2...v_nv_1$ in Figure \ref{fig:disk_t}). We reach a contradiction by calculating the Euler characteristic of $D$. 

	\begin{figure}[htbp]
		\centering
		\includegraphics{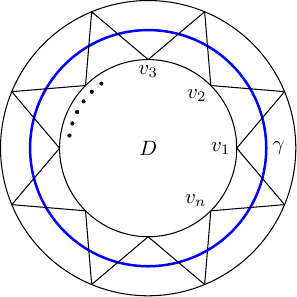}
		\caption{The curve $\gamma$ and the triangulated disk $D$}
		\label{fig:disk_t}
	\end{figure}

	Denote the sets of vertices, edges, and faces of $D$ by $V$, $E$, and $F$, respectively. Among all vertices, there are $n$ vertices $\{ v_1, v_2, ..., v_n \}$ in $ \partial D$ and $|V|-n$ vertices in the interior of $D$. Hence, the number of edges in $ \partial D$ is also $n$. 
	The Euler characteristic of $D$ is 
	\begin{equation}
		|V|-|E|+|F| = \chi(D) = 1. 
		\label{for:euler}
	\end{equation}

	Take two copies of $D$ and glue them along their boundaries, then we get a triangulated sphere $S^2$. The number of vertices, edges, and faces of $S^2$ is $2(|V|-n)+n$, $2(|E|-n)+n$ and $2|F|$ respectively. Since $S^2$ is a  triangulated closed surface, 
	\begin{eqnarray}
		3(2|F|) &=& 2(2(|E|-n) + n) \nonumber \\
		\text{namely, \, }3|F|  &=& 2|E| - n. \label{for:edge_face}
	\end{eqnarray}
	Let $\deg(v)$ be the degree of a vertex $v\in V$. Then 
	\begin{equation}
		\sum_{v\in V} \deg(v) = 2|E|. 
		\label{for:vertex_edge}
	\end{equation}
	Inserting (\ref{for:edge_face}) and (\ref{for:vertex_edge}) into (\ref{for:euler}), we obtain
	\begin{equation}
		|V| - \frac{1}{3} n - \frac{1}{6}\sum_{v\in V} \deg(v) = 1. \label{for:final_euler}
	\end{equation}

	By our assumption, we know that for any vertex $v$ in the interior of $D$, $\deg(v) \ge 6$. For a vertex $v_i$ in $ \partial D$, if we treat it as a vertex of the graph $ \widetilde{\mathcal{T}}$, then it has degree at least $6$. But two edges in $ \widetilde{\mathcal{T}}$ at $v_i$ are outside of $D$, therefore, in $D$, $\deg(v_i) \ge 4$ for $v_i \in \partial D$. 
	Then
	\begin{eqnarray*}
		1 &=& |V| - \frac{1}{3} n - \frac{1}{6}\sum_{v\in V} \deg(v) \\
		  &=&|V| - \frac{1}{3} n - \frac{1}{6}\sum_{v\in \partial D} \deg(v) - \frac{1}{6}  \sum_{v\in D \backslash\partial D} \deg(v) \\
		  & \le & |V| - \frac{1}{3} n - \frac{1}{6} \cdot(4n) - \frac{1}{6}\cdot 6(|V| - n) \\
		  &=&0,
	\end{eqnarray*}
	which leads to a contradiction. 

	Suppose that a curve $\gamma\subset \mathcal{D}( \widetilde{\mathcal{T}})$ bounds a monogon. Let $D$ be the triangulated disk consisting of triangles inside the monogon. Again we use \eqref{for:final_euler} and the lower bounds of the degree of vertices in $D$ to get a contradiction. 
	According to the directions of the edge of $ \widetilde{\mathcal{T}}$ passing through the intersection point of the monogon, we classify $D$ into two types shown in Figure \ref{fig:monogon_t}. 

\begin{figure}[htbp]
	\centering
	\begin{subfigure}[htbp]{.45\textwidth}
	\begin{center}
		\includegraphics{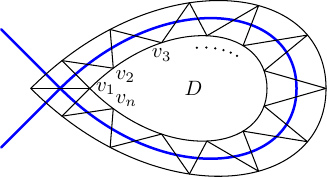}
	\end{center}
	\caption{Case 1}
	\label{fig:monogon_t_1}
	\end{subfigure}
	\begin{subfigure}[htbp]{.45\textwidth}
	\begin{center}
		\includegraphics{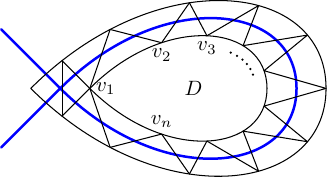}
	\end{center}
	\caption{Case 2}
	\label{fig:monogon_t_2}
	\end{subfigure}
	
	\caption{The monogon and the trinagulated disk $D$}
	\label{fig:monogon_t}
\end{figure}

In case 1 (see Figure \ref{fig:monogon_t_1}), there are $|V| -n$ vertices in the interior of $D$ and therefore each of them has degree at least $6$. In $ \partial D$, $v_2, v_3,..., v_n$ have degree at least $4$, and $v_1$ has degree at least $3$. Therefore,  it follows from \eqref{for:final_euler} that
	\begin{eqnarray*}
		1 &=& |V| - \frac{1}{3} n - \frac{1}{6}\sum_{v\in V} \deg(v) \\
		  &=&|V| - \frac{1}{3} n  - \frac{1}{6}  \sum_{v\in D \backslash\partial D} \deg(v) - \frac{1}{6}\sum_{i=2}^n \deg(v_i) - \frac{1}{6}\deg(v_1) \\
		  & \le & |V| - \frac{1}{3} n  - \frac{1}{6}\cdot 6(|V| - n) - \frac{1}{6} \cdot4(n-1) - \frac{1}{6} \cdot3 \\
		  &=& \frac{1}{6}, 
	\end{eqnarray*}
	which leads to a contradiction. 

	In case 2 (see Figure \ref{fig:monogon_t_2}), there are $|V| -n$ vertices in the interior of $D$ and therefore each of them has degree at least $6$. In $ \partial D$, $v_2, v_3,..., v_n$ have degree at least $4$, and $v_1$ has degree at least $2$. Therefore, again by \eqref{for:final_euler} we have
	\begin{eqnarray*}
		1 &=& |V| - \frac{1}{3} n - \frac{1}{6}\sum_{v\in V} \deg(v) \\
		  &=&|V| - \frac{1}{3} n  - \frac{1}{6}  \sum_{v\in D \backslash\partial D} \deg(v) - \frac{1}{6}\sum_{i=2}^n \deg(v_i) - \frac{1}{6}\deg(v_1) \\
		  & \le & |V| - \frac{1}{3} n  - \frac{1}{6}\cdot 6(|V| - n) - \frac{1}{6} \cdot4(n-1) - \frac{1}{6} \cdot2 \\
		  &=& \frac{1}{3},
	\end{eqnarray*}
	which also leads to a contradiction. Therefore, no curve in $\mathcal{D}( \widetilde{\mathcal{T}})$ bounds a monogon. 

	Suppose that a pair of curves in $\mathcal{D}( \widetilde{\mathcal{T}})$ bound a bigon. Let $D$ be the triangulated disk consisting of triangles inside the bigon. Again we use (\ref{for:final_euler}) and the lower bounds of degree of vertices in $D$ to get a contradiction. 
	According to the directions of the edges in $ \widetilde{\mathcal{T}}$ passing through the intersection points of the bigon, we classify $D$ into three types shown in Figure \ref{fig:bigon_t}. 

\begin{figure}[htbp]
	\centering
	\begin{subfigure}[htbp]{.45\textwidth}
	\begin{center}
		\includegraphics{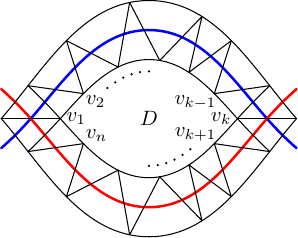}
	\end{center}
	\caption{Case 1}
	\label{fig:bigon_t_1}
	\end{subfigure}
	\begin{subfigure}[htbp]{.45\textwidth}
	\begin{center}
		\includegraphics{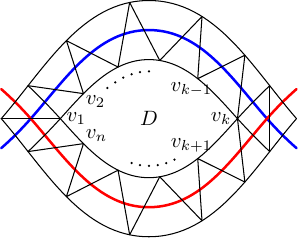}
	\end{center}
	\caption{Case 2}
	\label{fig:bigon_t_2}
	\end{subfigure}

	\begin{subfigure}[htbp]{.45\textwidth}
	\begin{center}
		\includegraphics{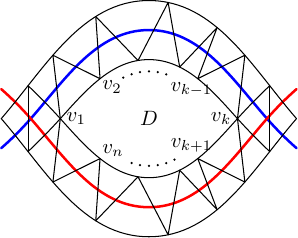}
	\end{center}
	\caption{Case 3}
	\label{fig:bigon_t_3}
	\end{subfigure}
	
	\caption{The bigon and the triangulated disk $D$}
	\label{fig:bigon_t}
\end{figure}	
In case 1 (see Figure \ref{fig:bigon_t_1}), there are $|V| -n$ vertices in the interior of $D$ and therefore each of them has degree at least $6$. In $ \partial D$, 
$v_1$ and $v_k$ have degree at least $3$, while other vertices in $\partial D$ have degree at least $4$. Therefore,
	\begin{eqnarray*}
		1 &=& |V| - \frac{1}{3} n - \frac{1}{6}\sum_{v\in V} \deg(v) \\
		  &=&|V| - \frac{1}{3} n  - \frac{1}{6}  \sum_{v\in D \backslash\partial D} \deg(v) - \frac{1}{6}\deg(v_1) - \frac{1}{6}\deg(v_k)- \frac{1}{6}\sum_{i\ne 1,k} \deg(v_i)  \\
		  & \le & |V| - \frac{1}{3} n  - \frac{1}{6}\cdot 6(|V| - n) - \frac{1}{6} \cdot3- \frac{1}{6} \cdot3- \frac{1}{6} \cdot4(n-2)  \\
		  &=& \frac{1}{3}, 
	\end{eqnarray*}
	which is a contradiction. 

	In case 2 (see Figure \ref{fig:bigon_t_2}), there are $|V| -n$ vertices in the interior of $D$ and therefore each of them has degree at least $6$. In $ \partial D$, 
$v_1$ has degree at least $3$ and $v_k$ has degree at least $2$, while other vertices in $\partial D$ have degree at least $4$. Therefore,
	\begin{eqnarray*}
		1 &=& |V| - \frac{1}{3} n - \frac{1}{6}\sum_{v\in V} \deg(v) \\
		  &=&|V| - \frac{1}{3} n  - \frac{1}{6}  \sum_{v\in D \backslash\partial D} \deg(v) - \frac{1}{6}\deg(v_1) - \frac{1}{6}\deg(v_k)- \frac{1}{6}\sum_{i\ne 1,k} \deg(v_i)  \\
		  & \le & |V| - \frac{1}{3} n  - \frac{1}{6}\cdot 6(|V| - n) - \frac{1}{6} \cdot3 - \frac{1}{6} \cdot2 - \frac{1}{6} \cdot4(n-2)  \\
		  &=& \frac{1}{2}, 
	\end{eqnarray*}
	which is a contradiction. 

	In case 3 (see Figure \ref{fig:bigon_t_3}), there are $|V| -n$ vertices in the interior of $D$ and therefore each of them has degree at least $6$. In $ \partial D$, 
$v_1$ and $v_k$ have degree at least $2$, while other vertices in $\partial D$ have degree at least $4$. Therefore,
	\begin{eqnarray*}
		1 &=& |V| - \frac{1}{3} n - \frac{1}{6}\sum_{v\in V} \deg(v) \\
		  &=&|V| - \frac{1}{3} n  - \frac{1}{6}  \sum_{v\in D \backslash\partial D} \deg(v) - \frac{1}{6}\deg(v_1) - \frac{1}{6}\deg(v_k)- \frac{1}{6}\sum_{i\ne 1,k} \deg(v_i)  \\
		  & \le & |V| - \frac{1}{3} n  - \frac{1}{6}\cdot 6(|V| - n) - \frac{1}{6} \cdot2- \frac{1}{6} \cdot2- \frac{1}{6} \cdot4(n-2)  \\
		  &=& \frac{2}{3}. 
	\end{eqnarray*}
	This is again a contradiction. Therefore, each pair of curves in $\mathcal{D}( \widetilde{\mathcal{T}})$ does not bound any bigon. 

	The proof is complete. 
\end{proof}
\begin{remark}
Here, the degree of triangulation $\mathcal{T}$ may be explained as the curvature of the surface and $\mathcal{D}(\mathcal{T})$ may be explained as geodesics in this curvature. An naive example is given by the triangulation of $\mathbb{R}^2$ by regular triangles whose degree is $6$ at every vertex and whose dual $4$-valent graph consists of geodesics in $\mathbb{R}^2$. Another example is about $\mathbb{H}^2$. In $\mathbb{H}^2$, regular triangulation always has degree $>6$, and its dual $4$-valent graph also consists of geodesics in $\mathbb{H}^2$. 
\end{remark}

Similarly to Lemma \ref{lem:mid_gph_incontractible}, Lemma \ref{lem:mid_gph_min_pos} and Proposition \ref{prop:fill_geod}, we also show the following results.

\begin{lemma}
	Each closed curve in $\mathcal{D}(\mathcal{T})$ is homotopically non-trival. 
	\label{lem:mid_gph_incontractible_tri}
\end{lemma}
\begin{proof}
By using Proposition \ref{prop:ess_tri_mid_gph_univ} instead of Proposition \ref{prop:ess_mid_gph_univ}, then the proof is exactly the same as that of Lemma \ref{lem:mid_gph_incontractible}.
\end{proof}

\begin{lemma}
The dual $4$-valent graph $\mathcal{D}(\mathcal{T})$ of $\mathcal{T}$ is in minimal position. 
	\label{lem:mid_gph_min_pos_tri}
\end{lemma}
\begin{proof}
	By Lemma \ref{lem:mid_gph_incontractible_tri}, each closed curve in $\mathcal{D}(\mathcal{T})$ is homotopically non-trival. By using Proposition \ref{prop:ess_tri_mid_gph_univ} instead of Proposition \ref{prop:ess_mid_gph_univ}, then the proof is exactly the same as that of Lemma \ref{lem:mid_gph_min_pos}.
\end{proof}

\begin{proposition}
The dual $4$-valent graph $\mathcal{D}(\mathcal{T})$ is isotopic to a filling set of closed geodesics on $X$. 
	\label{prop:fill_geod_tri}
\end{proposition}
\begin{proof}
The proof is exactly the same as that of Proposition \ref{prop:fill_geod}.
\end{proof}

Now we are ready to finish the proof of Proposition \ref{prop:ess_tri_mid_gph}.
\begin{proof}[Proof of Proposition \ref{prop:ess_tri_mid_gph}]
The result clearly follows from Lemma \ref{lem:mid_gph_incontractible_tri}, Lemma \ref{lem:mid_gph_min_pos_tri} and Proposition \ref{prop:fill_geod_tri}.
\end{proof}

\section{Proof of Theorem \ref{mt-1}}\label{sec-t-1}
In this section, we finish the proof of Theorem \ref{mt-1}. 

Let us first prepare the following elementary inequality that will be applied later.
	\begin{lemma}
		Let 
		\[
		f(x) = \arcsinh \left( \frac{1}{\sinh \left( \frac{x}{2} \right) } \right) - \frac{1}{2}\log 2.
		\]
		Then 
		\[
			\log\left(\frac{1}{x}\right) \le f(x) \le 
		\begin{cases}
			2, &\text{ if } x > \frac{1}{2}  ;\\
			3\log\left(\frac{1}{x}\right)&\text{ if }0 < x \le \frac{1}{2}.
		\end{cases}
		\]
		\label{lem:dist}
	\end{lemma}
	\begin{proof}
		Recall that $(\arcsinh(x))' = \frac{1}{\sqrt{1+x^2}}$. So  
		\[
			f'(x) = - \frac{1}{2\sinh \frac{x}{2}} <0. 
		\]
Since 
\[f\left( \frac{1}{2} \right) \approx 1.73<2<3\log 2\approx 2.08,\]
the upper bound follows from the following fact $$\left(3\log\left(\frac{1}{x}\right)\right)'-f'(x)=-\frac{3}{x}+ \frac{1}{2\sinh \frac{x}{2}}<0.$$ 
For the lower bound, note that 
		\[
			\left(\log\left(\frac{1}{x}\right)\right)' = - \frac{1}{x} <  - \frac{1}{2\sinh \frac{x}{2}} =f'(x). 
		\]
		The conclusion then follows because
        \[\lim\limits_{x\to 0^+}\left(f(x)- \log\left(\frac{1}{x}\right)\right)=\frac{3\log 2}{2}>0.\] 
        The proof is complete.
	\end{proof}

For any $X\in \sM_g$, set
\begin{equation}
    \mathrm{R}(X) = \sum_{\gamma \in \mathcal{P}(X), \ \ell(\gamma) <1} \log \left(\frac{1}{\ell(\gamma)}\right)
\end{equation}
where $\mathcal{P}(X)$ is the set of closed geodesics in $X$.

We first show the lower bound of Theorem \ref{mt-1}.
\begin{proposition}\label{mt-lb}
Let $X\in \sM_g$ be a closed hyperbolic surface. Then
\[\slfs(X)>\pi(g-1)+\mathrm{R}(X).\]
\end{proposition}
\begin{proof}
	For any filling closed multi-geodesic, it must transversely intersect with every simple closed geodesic of length $<1$ if it exists. Then it follows from the Collar Lemma (see \emph{e.g.} \cite[Theorem 4.4.6]{buser2010geometry}) and Lemma \ref{lem:dist} that
\begin{equation}\label{l-i-1}
   \slfs(X)>2\sum_{\gamma \in \mathcal{P}(X), \ \ell(\gamma) <1}\arcsinh \left( \frac{1}{\sinh \left( \frac{1}{2}\ell(\gamma) \right) } \right)>2\mathrm{R}(X).
\end{equation}
Meanwhile, for any $G\in \mathcal{G}$, $\ell(G) > \frac{1}{2}\area(X) = 2\pi(g-1)>0$ since $G$ is filling and the perimeter of a bounded hyperbolic contractible domain is larger than the area of this domain (see \emph{e.g.} \cite[Section 8.1 Page 211]{buser2010geometry}). Thus, 
    \begin{equation}\label{lb-ml}
       \inf\limits_{G\in \mathcal{G}} \ell(G) \geq 2\pi(g-1)> 0.   
    \end{equation}
 Therefore, 
	\begin{equation}\label{for:length_lower_g-1}
		\slfs(X) \geq 2\pi(g-1).
	\end{equation}
     The conclusion then follows from taking an average of \eqref{l-i-1} and \eqref{for:length_lower_g-1}.
\end{proof}
\begin{remark}
If we use \cite[Theorem 1.1]{gao2025shortest} instead of \eqref{lb-ml}, we may conclude that for any $X\in \sM_g$ and large enough $g$,
\[\slfs(X)>\frac{7}{2}g+\mathrm{R}(X).\]
\end{remark}

Next, we prove the upper bound in Theorem \ref{mt-1} using the results in the last section. We first show that 
\begin{proposition}\label{prop:length_upper}
	For a given hyperbolic surface $X\in \mathcal{M}_g$, there is a filling graph $G\in  \mathcal{G}$ on $X$ such that 
	\[
		\ell(G) \le 150g+6\mathrm{R}(X).
	\]
\end{proposition}
\begin{proof}
	By Theorem \ref{thm:buser_trigon} we know that $X$ admits a trigon decomposition $\mathcal{T}$, in which each trigon satisfies specific bounds on areas and perimeters. We first estimate the number of trigons and the length of $\mathcal{T}$. Recall that $\area(X)=4\pi(g-1)$ and that each trigon in $\mathcal{T}$ has area at least $0.19$. Then the number of trigons in $\mathcal{T}$ is at most $$ \frac{4\pi}{0.19}(g-1) \le 66.14 g.$$ 
	For each trigon in $\mathcal{T}$, since every side has length at most $\log 4$, its perimeter is $ \le 3\log 4$. Thus, if we treat $\mathcal{T}$ as a graph in $X$, then
	\begin{equation}\label{fg-t-0}
	    \ell(\mathcal{T}) \le \frac{3 \log 4}{2} \times 66.14 g \le 138 g. 
	\end{equation}
For any annulus-shaped trigon in $\mathcal{T}$ as shown in Figure \ref{fig:trigon}, recall that from \eqref{rmk:trigonlength} we know that the distance between the simple closed geodesic side $\gamma$ and the vertex $A$ or $B$ is 
	\[
	 	d(A, \gamma) = d(B, \gamma) = \arcsinh \left( \frac{1}{\sinh \left( \frac{1}{2}\ell(\gamma) \right) } \right) - \frac{1}{2}\log 2.
	\]
	Therefore,  by Lemma \ref{lem:dist} we have
	\[
		d(A, \gamma) = d(B, \gamma) = f(\ell(\gamma))\le \max(-3\log(\ell(\gamma)), 2). 
	\]

\begin{figure}[htbp]
	\centering
	\begin{subfigure}[htbp]{.45\textwidth}
	\begin{center}
		\includegraphics{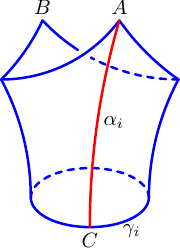}
	\end{center}
	\caption{Trigon $T_i$}
	\label{fig:trigon_2}
	\end{subfigure}
	\begin{subfigure}[htbp]{.45\textwidth}
	\begin{center}
		\includegraphics{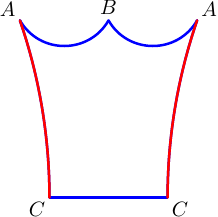}
	\end{center}
	\caption{Pentagon $T_i \backslash \alpha_i$}
	\label{fig:trigon_3}
	\end{subfigure}
	
	\caption{The annulus-shaped trigon $T_i$ and the pentagon $T_i \backslash \alpha_i$}
	\label{fig:trigon_cut}
\end{figure}	

\noindent Let $\left\{ \gamma_1, \gamma_2, ..., \gamma_n \right\} $ be the set of simple closed geodesics in $X$ with length $ \le \log 4$. Among them, assume that $ \gamma_1, \gamma_2, ..., \gamma_k $ have length $ < 1$; $ \gamma_{k+1}, \gamma_2, ..., \gamma_n  $ have length between $1$ and $\log 4$. Let $T_1, ..., T_{2n}$ be the annulus-shaped trigons in $X$, $T_{2i-1}$ and $T_{2i}$ form a neighborhood of $\gamma_i$. Let $\alpha_i$ be one of the two geodesic segments realizing the distance between the simple closed geodesic side of $T_i$ and one of the two vertices of $T_i$, for $i=1,2, ..., 2n$ (see Figure \ref{fig:trigon_cut} for an illustration). Then we define the graph $G$ as the union of $\mathcal{T}$ and the set of geodesic arcs $\left\{ \alpha_1, \alpha_2, ..., \alpha_{2n} \right\} $. It is clear that $G$ is a filling graph since every annulus-shaped trigon $T_i$ is cut by the corresponding geodesic segment $\alpha_i$ into a topological disk. Now we estimate the length of $G$. Write
\[\ell(G)= \ell(\mathcal{T}) + \sum_{i=1}^{2n} \ell(\alpha_i)=\ell(\mathcal{T}) + \sum_{i=1}^{2k} \ell(\alpha_i)+ \sum_{i=2k+1}^{2n} \ell(\alpha_i).\]
For $i > 2k$, $\ell(\alpha_i) \le f(1) \le  1.07$. So we have
\begin{equation}\label{fg-t-1}
\begin{aligned}
   \sum_{i=2k+1}^{2n} \ell(\alpha_i)\leq  2\times 1.07(n-k) =2.14(n-k).
\end{aligned}
\end{equation}
For $i \le 2k$, by Lemma \ref{lem:dist}, $\ell(\alpha_i) \le \max\left\{-3\log\left(\ell\left(\gamma_{\lceil \frac{i}{2} \rceil}\right)\right), 2\right\}.$ So we have 
\begin{equation}\label{fg-t-2}
\begin{aligned}
   \sum_{i=1}^{2k} \ell(\alpha_i)\leq &  2 \sum_{i=1}^{k} \max\left\{3\log\left(\frac{1}{\ell(\gamma_i)}\right), 2\right\}\leq 2 \sum_{i=1}^{k} \left(3\log\left(\frac{1}{\ell(\gamma_i)}\right)+2\right)\\
   &=6\times \left(\sum_{i=1}^{k} \log\left(\frac{1}{\ell(\gamma_i)}\right)\right)+4k.
\end{aligned}
\end{equation}
Using $k\leq n \leq (3g-3)$, it then follows from \eqref{fg-t-0}, \eqref{fg-t-1} and \eqref{fg-t-2} that
\[\ell(G)\le 138g+6\times \left(\sum_{i=1}^{k} \log\left(\frac{1}{\ell(\gamma_i)}\right)\right)+12(g-1)\le 150g+6\mathrm{R}(X).\]
The proof is complete.
\end{proof}

Now we are ready to prove Theorem \ref{mt-1}.
\begin{proof}[Proof of Theorem \ref{mt-1}]
From Proposition \ref{mt-lb} it suffices to show the upper bound. By Proposition \ref{prop:length_upper}, there is a filling graph $G \in \mathcal{G}$ such that $$\ell(G) \le 150g+6\mathrm{R}(X).$$ Then we choose a minimal $G_0 \in \mathcal{G}$ with $\ell(G_0) = \min\limits_{G\in \mathcal{G}}\ell(G)$ and $\mathcal{D}(G_0)$ to be the dual $4$-valent graph of $G_0$. By Proposition \ref{prop:ess_mid_gph}, $\mathcal{D}(G_0)$ is a filling closed multi-curve. By Proposition \ref{prop:graph_estimate}, we obtain 
	\[
		\slfs(X)\leq\ell(\mathcal{D}(G_0)) \le 2 \ell(G_0) \le 300g+12\mathrm{R}(X)
	\]
    as desired. The proof is complete.
\end{proof}

\section{Behavior of $\slfs$ on random hyperbolic surfaces}\label{s-rs}
In this section, we study the asymptotic behavior of $\slfs$ for random hyperbolic surfaces. In the first subsection, we prove Theorems \ref{thm:wp_prob} and \ref{thm:wp_exp}, using Theorem \ref{mt-1}; in the second subsection, we prove Theorem \ref{mt-3}, using Proposition \ref{prop:ess_tri_mid_gph_univ}.

\subsection{Weil-Petersson random surfaces} Let $\mathcal{M}_g$ be the moduli space of hyperbolic surfaces of genus $g$. For $t\in \mathbb{R}^n_{ > 0}$, let $\mathcal{M}_{g,n}(t)$ be the moduli space of hyperbolic surfaces of genus $g$ with $n$ geodesic boundaries of length $t$. Let $\mathcal{M}_{g,n}$ or $\mathcal{M}_{g,n}(0)$ be the moduli space of hyperbolic surfaces of genus $g$ with $n$ cusps. The volumes of these three moduli spaces in Weil-Petersson metric are denoted as $V_g$, $V_{g,n}(t)$ and $V_{g,n}$ respectively. 
For a Borel subset $B\subset \mathcal{M}_g$, we write 
\[
	\prob_{\mathrm{WP}}^g(B) = \frac{\vol(B)}{V_g}. 
\]
For $f: \mathcal{M}_g \to \mathbb{R}$, we also write
\[
	\mathbb{E}_{\mathrm{WP}}^{g}[f] = \frac{1}{V_g} \int_{\mathcal{M}_g} f \mathrm{d}\vol. 
\]

\noindent In this subsection, we use the same notation as in \cite{mirzakhani2013growth, WX22-GAFA, nie2023large}. We also use the following notation: for two positive functions $f_1$, $f_2$ of the variable $g$, say $$f_1(g) = O(f_2(g))$$ if there is a constant $C > 0$, independent of $g$, such that $f_1(g) \le C f_2(g)$. 
For two positive functions $f_{1,p}$ and $f_{2,p}$ about the parameter $p$ and the variable $g$, we say $$f_{1,p}(g) = O_p(f_{2,p}(g))$$ if there is a constant $C(p) > 0$, only depending on $p$, such that $f_{1,p}(g) \le C(p) f_{2,p}(g)$. We say $$f_{1,p}\asymp_p f_{2,p}$$ if $f_{1,p}(g) = O_p(f_{2,p}(g))$ and $f_{2,p}(g) = O_p(f_{1,p}(g))$. 

Here is an elementary inequality that will be applied later. 
For $a_1, a_2, ..., a_n >0$ and $p>0$, 
\begin{equation}
	\left( \frac{a_1+a_2+...+a_n}{n} \right)^p \le \left( \max_{1 \le i \le n} a_i \right)^p \le a_1^p+a_2^p + ...+ a_n^p.	
	\label{for:sum_an}
\end{equation}

The first lemma is about the expectation of $\mathrm{R}(X) = \sum\limits_{\gamma \in \mathcal{P}(X), \ \ell(\gamma) <1} \log \left(\frac{1}{\ell(\gamma)}\right)$. 
 \begin{lemma}
	For $p>0$, 	
	\[
	\frac{\mathbb{E}_{\mathrm{WP}}^{g}[\mathrm{R}(X)^p]}{g^p} =  O_p(1). 
	\]
    If $p=1$,	\[
	\mathbb{E}_{\mathrm{WP}}^{g}[\mathrm{R}(X)] = O(1). 
	\]
	\label{lem:r_exp}
\end{lemma}

\begin{proof}
	Define $f_p: \mathbb{R}_{>0} \to \mathbb{R}$ as 
\[
	f_p(x) :=
	\begin{cases}
		(-3\log x)^p, &\text{ if }0 <x< 1;\\
		0, & \text{ if } x \ge  1 . 
	\end{cases}
\]
And we define $f^{\gamma}_p: \mathcal{M}_g \to \mathbb{R}$ as 
\[
	f^{\gamma}_p(X) := \sum\limits_{\gamma \in \mathcal{P}(X), \ \ell(\gamma) <1} f_p \left( \ell_{\gamma}(X) \right). 
\]

\noindent By definition, we know that $$\mathrm{R}(X) \le f^{\gamma}_1(X).$$ By the classical Collar Lemma we know that the number of closed geodesics of length $\leq 1$ is no more than $3g-3$, it then follows from \eqref{for:sum_an} that $$\mathrm{R}(X)^p \le (3g)^p f^{\gamma}_p(X)$$ for all $p>0$. Now we estimate $\mathbb{E}_{\mathrm{WP}}^{g}[f^{\gamma}_p]$. 
Define $V(\text{cut},t)$ to be the volume of the moduli space $\{X \backslash \gamma; \ X\in \mathcal{M}_g, \gamma\subset X \text{ is a simple closed geodesic with }\ell_{\gamma}(X)=t \}$. Then
\[
	V(\text{cut},t) = V_{g-1,2}(t,t) + \sum_{k=1}^{g-1} V_{k,1}(t)\times V_{g-k,1}(t).
\]
If $t=0$, it is known from \cite[Theorem 3.5]{mirzakhani2013growth} that $V_{g-1,2}\asymp V_g$, and it is also known from \cite[Corollary 3.7]{mirzakhani2013growth} that $\sum_{k=1}^{g-1} V_{k,1}\times V_{g-k,1} \asymp \frac{V_g}{g}$. So 
\begin{equation}\label{v-e-c}
\left(V_{g-1,2} + \sum_{k=1}^{g-1} V_{k,1}\times V_{g-k,1}\right)\asymp V_g.  
\end{equation}
\noindent Now we apply Mirzakhani's integration formula \cite[Theorem 7.1]{mirzakhani2007simple} to bound $\mathbb{E}_{\mathrm{WP}}^{g}[f^{\gamma}_p]$. 
\begin{eqnarray*}
	&&\int_{\mathcal{M}_g} f_p^{\gamma}(X) \mathrm{d}\vol \\
	&\leq &\int_{0}^{+\infty} tf_p(t) V(\text{cut},t)\mathrm{d}t \text{\,\,\,\quad (by Mirzakhani's integration formula)}\\
							  &=& \int_0^{1}tf_p(t) \left( V_{g-1,2}(t,t) + \sum_{k=1}^{g-1} V_{k,1}(t)\times V_{g-k,1}(t) \right) \mathrm{d}t \\
							  &\asymp & \left( V_{g-1,2} + \sum_{k=1}^{g-1} V_{k,1}\times V_{g-k,1} \right)\int_0^{1}tf_p(t)\mathrm{d}t \text{\,\,\,\quad (by \cite[Lemma 22]{nie2023large})} \\
							  & \asymp &V_g \times \int_0^{1}tf_p(t)\mathrm{d}t  \text{\,\,\,\,\quad (by \eqref{v-e-c})}\\
							  & \asymp_p & V_{g} .
\end{eqnarray*}

\noindent Therefore, when $p=1$, 
\[
	\mathbb{E}_{\mathrm{WP}}^{g}[\mathrm{R}(X)] \le \mathbb{E}_{\mathrm{WP}}^{g}[f_1^{\gamma}] = O(1). 
\]
For $p > 0$, 
\[
	\frac{\mathbb{E}_{\mathrm{WP}}^{g}[\mathrm{R}(X)^p]}{g^p} \le (3)^p\times\mathbb{E}_{\mathrm{WP}}^{g}[f_p^{\gamma}] = O_p(1). 
\]
The proof is complete. 
\end{proof}

Now we are ready to prove Theorems \ref{thm:wp_prob} and \ref{thm:wp_exp}.

\begin{proof}[Proof of Theorem \ref{thm:wp_exp}]
By \cite[Theorem 1.1]{gao2025shortest} we know that for $g$ large enough, 
\begin{equation}\label{lb-7g}
\slfs(X)\geq (8g-4)\arccosh\left(\sqrt{2}\cos\left(\frac{\pi}{8g-4}\right)\right)>7g   
\end{equation}
for any $X\in \sM_g$. So for any $p>0$, the lower bound to $ \mathbb{E}_{\mathrm{WP}}^g[(\slfs)^p]$ follows directly:
\[
\mathbb{E}_{\mathrm{WP}}^g[(\slfs)^p] > \mathbb{E}_{\mathrm{WP}}^g \left[ \left(7g  \right)^p  \right] = \left(7g  \right)^p. 
\]
In particular, when $p=1$, 
\[
\mathbb{E}_{\mathrm{WP}}^g[\slfs] > 7g. 
\]
Recall that the upper bound in Theorem \ref{mt-1} says that for any $X \in \mathcal{M}_g$, 
	\[
		 \slfs(X) \le 300g+12\mathrm{R}(X).
	\]
\noindent  For the upper bound, when $p=1$, for sufficiently large $g$,
\begin{eqnarray*}
	\mathbb{E}_{\mathrm{WP}}^g[\slfs] &\le& \mathbb{E}_{\mathrm{WP}}^g \left[ \left(300g + 12\mathrm{R}(X)  \right)  \right]  \\
					      &=& \mathbb{E}_{\mathrm{WP}}^g \left[ 300g \right] + \mathbb{E}_{\mathrm{WP}}^g \left[12\mathrm{R}(X)    \right] \\
					      &=& 300g + O(1) \text{\,\,\,\quad (by Lemma \ref{lem:r_exp})} \\
					      &<& C_2g,
\end{eqnarray*}
where $C_2>300$ is arbitrary. For $p>0$, 
\begin{eqnarray*}
	\mathbb{E}_{\mathrm{WP}}^g[(\slfs)^p] &\le& \mathbb{E}_{\mathrm{WP}}^g \left[ \left(300g + 12\mathrm{R}(X)  \right)^p  \right]  \\
					      & \le & \mathbb{E}_{\mathrm{WP}}^g \left[ 2^p \left( \left(300g\right)^p + \left(12\mathrm{R}(X)  \right)^p \right)   \right] \text{\,\,\,\quad (by (\ref{for:sum_an}))} \\
					      &=& O_p(1)\times g^p + O_p(1)\times g^p \text{\,\,\,\quad \  (by Lemma \ref{lem:r_exp})} \\
					      &=& O_p(1)\times g^p. 
\end{eqnarray*}
The proof is complete. 
\end{proof}

\begin{proof}[Proof of Theorem \ref{thm:wp_prob}]
	For any $X \in \mathcal{M}_g$, by \eqref{lb-7g} we know that for $g$ large enough, $$\mathcal{L}_{\sys}^{\mathrm{fill}}(X) >7g.$$ It remains to show the upper bound. For $C > 300$, it follows that 
	\begin{eqnarray*}
	& &\prob_{\mathrm{WP}}^{g}(X\in \mathcal{M}_g; \ \mathcal{L}_{\sys}^{\mathrm{fill}}(X) \geq Cg) \\
	& \le & \prob_{\mathrm{WP}}^{g}\left(X \in \mathcal{M}_g; \ \mathrm{R}(X) \geq \frac{C-300}{12}g  \right) \text{\,\,\, \quad  (by Theorem \ref{mt-1}})\\
	& \le & \frac{\mathbb{E}^g_{\mathrm{WP}} [\mathrm{R}(X)]}{\frac{C-300}{12}g} \text{\,\,\,\quad \quad (by Markov's inequality)} \\
	& \le & \frac{O(1)}{\frac{C-300}{12}g} \text{\,\,\,\quad \quad \quad (by Lemma \ref{lem:r_exp})} \\
	&=& O \left( \frac{1}{g} \right). 
	\end{eqnarray*}
	The theorem then follows by letting $g\to \infty$. 	
\end{proof}

\subsection{Brooks-Makover random surfaces}
We refer to \emph{e.g.} \cite{brooks2004random, petri2015random, SW23} for the construction of random hyperbolic surfaces in the Brooks-Makover model. A cusped hyperbolic surface in such a model is required to have zero shearing coordinates. Two ideal triangles that share a common edge have shearing zero at the common edge if the inscribed circles of the two triangles intersect on the common edge (see Figure \ref{fig:ideal_1}). 
\begin{figure}[htbp]
	\centering
	\includegraphics{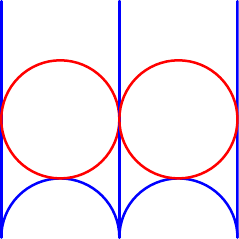}
	\caption{Zero shearing}
	\label{fig:ideal_1}
\end{figure}

We first prove Theorem \ref{mt-3} for the cusp case.

\begin{proposition}\label{mt3-cus}
For any $\omega\in \Omega_N$,
\[\pi N \le  \mathcal{L}^{\mathrm{fill}}_{\mathrm{sys}} (S_O(\omega)) \le 6N.\]
\end{proposition}

\begin{proof}
For the lower bound to $\mathcal{L}^{\mathrm{fill}}_{\mathrm{sys}} (S_O(\omega))$, observe that a filling closed multi-geodesic cuts $S_O(\omega)$ into geodesic polygons and geodesic polygons with one cusp. For a polygon $P$, with or without one cusp, $\area(P) \le \perim(P)$. For the polygon without any cusp, it follows from \emph{e.g.} \cite[Section 8.1 Page 211]{buser2010geometry}. For the polygon with one cusp, it follows from \emph{e.g.} \cite[Theorem 1.2]{saha2024length}. 
Therefore, $$\mathcal{L}^{\mathrm{fill}}_{\mathrm{sys}} (S_O(\omega)) \ge \frac{1}{2}\area(S_O(\omega)) = \pi N.$$
Next we prove the upper bound. For each ideal triangle, we join the three points on its three edges that touch the triangle's inscribed circle. Then we get a small geodesic triangle contained in the ideal triangle (see Figure \ref{fig:ideal_2}). Observe that by the symmetry of the triangles (see Figure \ref{fig:ideal_2}), $\angle 1 = \angle 3 = \angle 4 = \angle 6$, $\angle 2 = \angle 5$, therefore $\angle 2 + \angle 3 + \angle 4 = \angle 1 + \angle 6 + \angle 5 = \pi$ and $\angle 3+ \angle 4 + \angle 5 = \angle 6 + \angle 1 + \angle 2 = \pi$. Thus, the edges of small triangles form a set of closed geodesics in $S_O(\omega)$. 
\begin{figure}[htbp]
	\centering
	\includegraphics{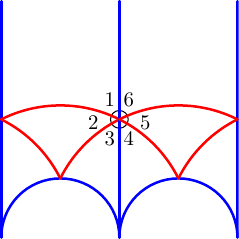}
	\caption{Small geodesic triangles in $S_O(\omega)$}
	\label{fig:ideal_2}
\end{figure}
\noindent The set of those closed geodesics fills $S_O(\omega)$, since their union cut the surface $S_O(\omega)$ into certain convex triangles and cusped disks. It remains to compute its total length. This can be obtained by elementary computations, which we briefly discuss here. 
\begin{figure}[htbp]
	\centering
	\includegraphics{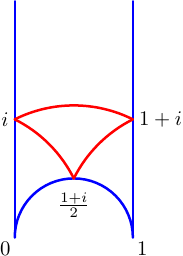}
	\caption{Length of a small triangle}
	\label{fig:ideal_3}
\end{figure}
First, we know that the length of each edge of a small triangle (see Figure \ref{fig:ideal_3}) is equal to $$l=2\arcsinh \frac{1}{2} \approx 0.962$$ by direct calculation (see \emph{e.g.} \cite[Remark on P. 15296]{SW23}). Then the perimeter of each small triangle is $3l \approx 2.887$. Therefore, we get
\[\mathcal{L}^{\mathrm{fill}}_{\mathrm{sys}} (S_O(\omega))\leq 3l\cdot 2N \approx 5.775N < 6N\]
as desired. The proof is complete.
\end{proof}

Now we prove Theorem \ref{mt-3} for the closed case.
\begin{proposition}\label{mt3-comp}
The following limit holds.
\[
		\lim\limits_{N\to \infty}\prob_{\mathrm{BM}}^N \left( \omega\in \Omega_N; \ \frac{7}{2}N <  \mathcal{L}^{\mathrm{fill}}_{\mathrm{sys}} (S_C(\omega)) < 6N \right) = 1.
	\]
\end{proposition}
\begin{proof}
For the lower bound, it is known by \cite[Corollary $5.1$ and $(2.1)$]{gamburd2006poisson} that
\[
	\lim\limits_{N\to \infty}\prob_{\mathrm{BM}}^N \left( \omega\in \Omega_N; \  \mathrm{genus}(S_C(\omega)) = 1+ \frac{N}{2} + O(\log N) \right)=1.
	\]
And it is also known from \eqref{lb-7g} that for sufficiently large $g$, 
\[\mathcal{L}^{\mathrm{fill}}_{\mathrm{sys}} (X_g)>7g.\]
Combining these two results above, we get 
\[\lim\limits_{N\to \infty}\prob_{\mathrm{BM}}^N \left( \omega\in \Omega_N; \   \mathcal{L}^{\mathrm{fill}}_{\mathrm{sys}} (S_C(\omega))>\frac{7}{2}N \right) = 1.\]
For the upper bound, notice that the ideal triangulation of $S_O(\omega)$ induces a geodesic triangulation $\mathcal{T}$ of $S_C(\omega)$ (take the image of the ideal triangulation of $S_O(\omega)$ onto $S_C(\omega)$ and straighten their edges). So, the set of vertices of $\mathcal{T}$ consists of all the punctures of $S_O(\omega)$, and the degree of each vertex depends on the size of the horoball around the corresponding puncture in $S_O(\omega)$. Thanks to 
\cite[Theorem 2.1]{brooks2004random}, we know that for any $L>0$,
	\[\lim\limits_{N\to \infty}\prob_{\mathrm{BM}}^N \left(\omega\in \Omega_N; \ S_O(\omega) \text{ has large cusps} \ge L \right)=1.
	\]
Thus, choosing $L>0$ large enough, one may assume that for a random surface in $\Omega_N$, each vertex of the geodesic triangulation $\mathcal{T}$ of $S_C(\omega)$ has degree $>6$. Then it follows from Proposition \ref{prop:ess_tri_mid_gph} that the dual $4$-valent graph $\mathcal{D}(\mathcal{T})$ of $\mathcal{T}$ is a filling closed multi-curve, isotopic to a filling closed multi-geodesic in $S_C(\omega)$. Observe that the image of the filling closed multi-geodesic in $S_O(\omega)$ constructed in Proposition \ref{mt3-cus} is isotopic to $\mathcal{D}(\mathcal{T})$ in $S_C(\omega)$, since the inclusion $i:S_O(\omega) \hookrightarrow S_C(\omega)$ is holomorphic, it then follows from the standard Schwarz-Pick Lemma that the length of the filling closed multi-geodesic that is isotopic to $\mathcal{D}(\mathcal{T})$ in $S_C(\omega)$ is no more than the length of the filling closed multi-geodesic in $S_O(\omega)$ constructed in Proposition \ref{mt3-cus}, that is less than $6N$. This proves the upper bound. The proof is complete.
\end{proof}

\begin{proof}[Proof of Theorem \ref{mt-3}]
It follows from Propositions \ref{mt3-cus} and \ref{mt3-comp}.
\end{proof}

\begin{remark}
In general, the image of curves in minimal position in $S_O(\omega)$ may not again be in minimal position in $S_C(\omega)$, for example, see Figure \ref{fig:bigon}. Proposition \ref{prop:ess_tri_mid_gph} is essential in the proof of Proposition \ref{mt3-comp}, thus
the upper bound is only for random surfaces in $\Omega_N$, rather than for every surface in $\Omega_N$.    
\end{remark}

 \bibliographystyle{alpha}
 \bibliography{random_fill}
\end{document}